\documentclass[12pt]{amsart}
\usepackage{amscd}
\usepackage{amssymb}
\usepackage{a4wide}
\usepackage{amstext}
\usepackage{amsthm}
\usepackage{mathrsfs}
\usepackage[final]{hyperref}
\hypersetup{unicode= false, colorlinks=true, linkcolor=blue,
anchorcolor=blue, citecolor=green, filecolor=red, menucolor=blue,
pagecolor=blue, urlcolor=blue} 
\numberwithin{equation}{section}

\newcommand{\CC}{\mathbb {C}}

\renewcommand{\phi}{\varphi}

\newcommand{\ima}{{\rm Im}\,}

\newcommand{\co}{\mathbb{C}}

\newcommand{\zl}{\mathbb{Z}}
\newcommand{\zz}{\mathcal{Z}}
\newcommand{\cp}{\mathbb{C^+}}

\newcommand{\rl}{\mathbb{R}}

\newcommand{\fw}{\mathcal{F}_W}
\newcommand{\fa}{\mathcal{F}_\alpha}
\newcommand{\hh}{\mathcal{H}}

\newcommand{\pp}{\mathcal{P}}

\newcommand{\MM}{\mathcal{M}}
\newcommand{\HH}{\mathcal{H}}

\newcommand{\Ordo}{\mathrm{O}}
\newcommand{\ordo}{\mathrm{o}}
\newcommand{\e}{\mathrm{e}}
\newcommand{\imag}{\mathrm{i}}
\newcommand{\Lip}{\mathrm{Lip}}

\newtheorem{Thm}{Theorem}[section]
\newtheorem{theorem}[Thm]{Theorem}
\newtheorem{lemma}[Thm]{Lemma}
\newtheorem{proposition}[Thm]{Proposition}
\newtheorem{corollary}[Thm]{Corollary}
\newtheorem{remark}[Thm]{Remark}
\newtheorem{example}[Thm]{Example}

\begin{document}
\sloppy
\title[Backward shift and nearly invariant subspaces of Fock-type spaces]
{Backward shift and nearly invariant subspaces \\ of Fock-type spaces}
\author{Alexandru Aleman, Anton Baranov, Yurii Belov, H{aa}kan Hedenmalm}

\address{Alexandru Aleman
\newline  Centre for Mathematical Sciences, Lund University, Sweden
\newline {\tt alexandru.aleman@math.lu.se}
\smallskip
\newline \phantom{x}\,\, Anton Baranov
\newline Department of Mathematics and Mechanics, St.~Petersburg State University,
St.~Petersburg, Russia,
\newline {\tt anton.d.baranov@gmail.com}
\smallskip
\newline \phantom{x}\,\, Yurii Belov
\newline  Department of Mathematics and Computer Science, St.~Petersburg State
University, St. Petersburg, Russia,
\newline {\tt j\_b\_juri\_belov@mail.ru}
\smallskip
\newline \phantom{x}\,\, H{aa}kan Hedenmalm
\newline  Department of Mathematics, KTH Royal Institute of Technology,
Stockholm, Sweden,
\newline {haakanh@math.kth.se}
}

\thanks{The work is supported by the Russian Science Foundation grant 19-11-00058}

\begin{abstract}
We study the structure of the  backward shift invariant and nearly invariant
subspaces in weighted Fock-type spaces $\fw^p$,  whose weight is  not necessarily
radial. We show that in the spaces 
$\fw^p$ which contain the polynomials as a dense subspace (in particular, in the
radial case) all nontrivial backward shift invariant subspaces are of the form
$\mathcal{P}_n$, i.e., finite dimensional subspaces consisting of polynomials
of degree at most $n$. In general, the structure of the nearly invariant subspaces
is more complicated. In the case of spaces of slow growth (up to zero exponential
type) we establish an analogue of de Branges' Ordering Theorem. We then construct
examples which show that the result fails for general Fock-type spaces of larger
growth.
\end{abstract}

\maketitle

\section{Introduction}
\label{int}

\subsection{Backward shift invariant and nearly invariant subspaces}  
Shift invariant and nearly invariant subspaces form an important part of theory of
spaces of analytic functions.
The basic setting here is the Hardy space $H^2$ where the shift invariant
subspaces are described by the famous Beurling
theorem, while nearly invariant subspaces were studied in detail by Hayashi,
Hitt and Sarason \cite{hay, hitt, sar}. In the case of other 
classical spaces in the disc (Bergman, Dirichlet, etc.) the structure of the
shift and backward shift invariant subspaces is much more complicated and a
complete description seems to be out of reach
(see, e.g., \cite{risu, ar, arr}).

Now we recall the necessary definitions. Let $\Omega$ be a domain in $\mathbb{C}$
with $0\in \Omega$, and let $\HH$ be a Banach space of functions analytic
in $\Omega$ such that point evaluation functionals $f\mapsto f(w)$ are bounded
for any $w\in \Omega$. In what folllows we always assume that 
the space $\hh$ has the {\it division property}, that is,
$\frac{f(z)}{z-w} \in \hh$ whenever $f\in \hh$, $w\in \Omega$ and $f(w) = 0$.
We say that a closed linear subspace $\MM$ of $\HH$ is backward shift invariant
if $\frac{f(z) - f(0)}{z} \in \MM$ for any $f\in \MM$. In other words, $\MM$
is invariant
for the backward shift $L: f\mapsto \frac{f - f(0)}{z}$.

There exists the more general notion of a nearly invariant subspace.  A closed
linear subspace $\MM$ of $\hh$ is said to be {\it nearly invariant} if
$\frac{f(z)}{z} \in \MM$ whenever $f\in \MM$ and $f(0) =0 $. Clearly, if
the function which is identically equal to $1$ belongs to $\MM$,
then nearly invariance is equivalent to backward shift invariance. In general, 
$\MM$ is backward shift invariant if and only if $\MM$ is nearly invariant
and $1 \in \MM +z \MM$
(i.e., $1= f +z g$ for some $f,g\in \MM$). Naturally, this condition rules out
the possibility that there would exist a sequence tending to infinity such that
all elements of a backward shift invariant subspace
$\MM$ would decay like $\ordo(|z|^{-1})$ along that sequence.

The choice of the point $0$ is not essential. For any $w\in \Omega$
such that there exists $f\in \MM$ with $f(w)\ne 0$
(i.e., $w$ is not a {\it common zero for $\MM$}), one has the implications
$$
\MM \quad \text{is backward shift invariant,} \quad 
f\in \MM, \ \ f(w) = 0 \ \Longrightarrow \ 
\frac{f(z)-f(w)}{z-w}\in \MM,
$$
and
$$
\MM \quad \text{is nearly invariant,} \quad 
f\in \MM, \ \ f(w) = 0 \ \Longrightarrow \ 
\frac{f(z)}{z-w}\in \MM.
$$
In the context of Hardy spaces in general domains the equivalence
of near invariance and division invariance is shown 
in \cite[Proposition 5.1]{ar}; a similar argument works 
for general Banach spaces of analytic functions \cite[Proposition 7.1]{by1}.

While backward shift invariant subspaces never have common zeros, this might
happen for nearly invariant subspaces   
of a space of analytic functions with the division property. Throughout in
this paper we shall consider 
nearly invariant subspaces {\it without common zeros.}

Let us make the following simple observations. If our space $\HH$ contains the set
$\pp$ of all polynomials, then any subspace of the form $\pp_n$ (consisting of
all polynomials of degree at most $n$) is a backward shift invariant subspace. 
As we will see,  for a class of weighted Fock-type spaces it is possible that
all nontrivial backward shift invariant subspaces and even all nearly-invariant
subspaces are of the form $\pp_n$. Note that in this case all nearly-invariant 
subspaces are ordered by inclusion (recall that an operator whose invariant
subspaces are ordered by inclusion is said to be {\it unicellular}).

It is however possible that nearly invariant subspaces have ordered structure
even in the case when they are not of the form $\pp_n$.
A model situation where this property holds is given by the Ordering Theorem
for de Branges spaces \cite{br, rom} (for a generalization to the
so-called Cauchy--de Branges spaces, see \cite{abb}). In the present paper
we study the structure of backward shift invariant
and nearly invariant subspaces for weighted Fock-type spaces of entire functions.
\medskip

\subsection{Weighted Fock-type spaces}
By a {\it weight}  we simply mean a positive  function $W$ in $\co$ which
is measurable with respect to planar Lebesgue measure $m_2$ in the complex
plane $\CC$.
With any weight $W$ and $p\in [1, +\infty)$ we associate the Fock-type space
of entire functions
\[
\fw^p = \bigg\{ F \in Hol(\co): \|F\|^p_{\fw^p} = \int_\co |F(z)|^p W(z) dm_2(z)
<+\infty  \bigg\}.
\]
In what follows we shall always assume that $W$ is bounded from above and
below by positive constants on any compact. In this case, the point
evaluation functionals are bounded on $\fw^p$ 
and $\fw^p$ is a Banach space with the division property. Conversely, it is
also easy to see that if  point evaluation functionals are bounded on $\fw^p$,
we can find a weight $\tilde{W}$ which is bounded above and below by positive
constants on any compact, and $\mathcal{F}_{\tilde{W}}^p=\fw^p$ with
equivalent norms.

In the case $p=2$ we will omit the exponent $p$.
Clearly, $\fw$ is a reproducing kernel Hilbert space. 
 
A typical example of a Fock-type space is a radial space $\fw^p$ with
$W(z) = \exp(-\varphi(|z|))$ 
where $\varphi$ is a function on $[0, +\infty)$ such that
$\log r = \ordo(\varphi(r))$, $r\to+\infty$ 
(to exclude the trivial finite-dimensional case). 
E.g., one can take $\varphi(r) = a r^\alpha$, $a, \alpha>0$; in what follows
we denote any corresponding space by $\fa^p$ (formally, it depends also on
$a$, but this dependence is not essential). The case $p=2$ and $\alpha=2$
corresponds to the classical Bargmann--Segal--Fock space,  ubiquitous in
applications -- from theoretical physics to time-frequency analysis.
The Hilbert spaces in this scale will be denoted by $\mathcal{F}_\alpha$.

In the present paper we will also consider non-radial Fock spaces.
It should be mentioned that Fock-type spaces are fairly general objects
which cover even seemingly unrelated examples.  
It was shown in \cite{bom} that any de Branges space (whose norm is given
by a weighted $L^2$-integral over the real axis) coincides with equivalence
of norms with some Fock-type space.
E.g., the Paley-Wiener space $PW_{[-a,a]}$, the image of $L^2([-a,a])$
by the Fourier transform,  coincides with the space $\fw$ where 
$W(z) = (1 +|\ima z|)^{-2}\e^{-2a|\ima z|}$ and, moreover, the Paley--Wiener space
is the only de Branges space which can be realized as a Fock-type space with the 
weight depending on $\ima z$ only. The same is true for a wider class of the
so-called Cauchy--de Branges spaces (for their theory see \cite{abb}): any 
Cauchy--de Branges can be realized as a Fock-type space. 

Note also that in the radial Hilbertian Fock spaces the monomials $\{z^n\}_{n\ge 0}$ 
form an orthogonal basis. Therefore, one can isometrically identify such spaces with
weighted spaces of sequences:
\[
\fw = \bigg\{ F(z) = \sum_{n\ge 0} c_n z^n: \ \|F\|^2 = 
\sum_{n\ge 0} W_n |c_n|^2< +\infty \bigg\}, \qquad W_n = 2\pi \int_0^\infty r^{2n+1} W(r)dr.
\]
\medskip

\subsection{Main results}
Given a Fock-type space $\fw^p$, we address the following questions: 
\medskip

a) When are all nontrivial backward shift invariant or nearly invariant subspaces  
of the form $\pp_n$ (recall that we consider only 
subspaces without common zeros)?
\medskip

b) When is the set of all nearly invariant subspaces  totally ordered by inclusion, 
i.e., given two nearly invariant subspaces $\MM_1$ and $\MM_2$ of
$\fw^p$, is it true that either $\MM_1\subset \MM_2$ or $\MM_2\subset \MM_1$?
\medskip

For backward shift invariant subspaces in radial Hilbertian Fock
spaces the answer to the first question is positive and essentially known. 
In this case the problem is equivalent to the unicellularity of the weighted shifts on the space $\ell^2$
studied, e.g., in \cite{nik, shi, ya1, ya2, dom}. 
One of the strongest results in this direction is due to D.V.~Yakubovich 
\cite{ya1, ya2} who showed (answering a question of A.L. Shields \cite{shi})
that if $(\lambda_n)_{n\ge 1}$ is a positive non-increasing sequence tending to zero, 
then the weighted shift 
$$
T: \ell^p(\zl_+) \to \ell^p(\zl_+), \qquad (c_0, c_1, \dots) \mapsto (\lambda_1 c_1, \lambda_2 c_2, \dots)
$$
is unicellular for any $p\in (1,\infty)$, and so its invariant subspaces are of the form $\{(c_n)_{n\ge 0}: \ c_n = 0,
\ n>N\}$ for some $N$. 
From this it is clear that  $L$ is unicellular on any radial Hilbertian space $\fw$, hence
all backward shift invariant subspaces of $\fw $ are of the form $\pp_n$.  However, it is not clear
whether one can relate the backward shift on $\fw^p$ with weighted shifts on sequence spaces
when $p\ne 2$.

Our first main result extends the above to a large class of  Fock-type spaces
$\fw^p$ ( $W$ not necessarily radial),  
containing  the set of polynomials  as a  dense subset. To state it we
introduce the following terminology. 
We say that $\fw^p$ is a space of finite order if any function $F\in\fw^p$
is of finite order. In fact, 
in this case there exists a uniform upper bound for the orders of elements
of $\fw^p$.
Analogously, we say that $\fw^p$ is a space of zero exponential type, if any
$F\in\fw^p$ 
is of zero type with respect to the order 1.

\begin{theorem} 
\label{main1}
Let $\fw^p$ be a space of finite order such that $\fw^p$ contains the set
$\pp$ of all 
polynomials as a dense subset.
Then any nontrivial backward shift invariant subspace is of the form
$\pp_n$ for some $n\in\zl_+$ and, thus, 
$L$ is unicellular on $\fw^p$.
\end{theorem} 

The following theorem shows that we can omit the restriction that $\fw^p$ is of
finite order
if the weight is radial. Note that in this case,  polynomials are dense in
$\fw^p$ whenever they are contained in it  (see Proposition
\ref{den} below).

\begin{theorem} 
\label{main11}
Let $\fw^p$ be a radial Fock-type space, $1<p<+\infty$, containing all polynomials. 
Then any nontrivial backward shift invariant subspace is of the form
$\pp_n$ for some $n\in\zl_+$ and, thus, $L$ is unicellular on $\fw^p$.
\end{theorem} 

We now turn to the ordering property for nearly invariant subspaces.
Here the  threshold is given by 
the order 1. Spaces of smaller order have the ordering property
while even in the standard radial 
Fock spaces of order 1, or higher,  nearly invariant subspaces are not
ordered by inclusion. 
The next theorem can be obtained  by a modification of the beautiful de Branges' 
proof of Ordering Theorem for de Branges spaces based
on a Phragm\'en--Lindel\"of type result due to M. Heins. 

\begin{theorem} 
\label{main3}
Let $\fw^p$ be a space of zero exponential type. 
Then the set of nearly invariant subspaces of $\fw^p$ is ordered by inclusion. 
\end{theorem} 

In the case when polynomials are dense in $\fw^p$ this leads to a complete
description of all nearly invariant subspaces. 

\begin{corollary}
\label{cor1}
Let $\fw^p$ be a  Fock-type space of zero exponential type which
contains the polynomials as a dense subspace. Then any nontrivial nearly invariant 
subspace is of the form $\pp_n$ for some $n\in\zl_+$. 
This is true, in particular, for all spaces $\fa^p$, $\alpha\in (0,1)$.
\end{corollary} 
 
Obviously, in the case $\alpha\ge 1$, the spaces $\fa^p$  contain
finite-dimensional nearly invariant subspaces of the form $\e^{Q}\mathcal{P}_n$ 
where $Q$ is a fixed polynomial of degree at most $\alpha$ and one cannot
expect the ordered structure.
However, the collection of nearly invariant subspaces is much larger.  
The following result is a special case of a more general construction (see
Sections \ref{barg} and \ref{count}). 

\begin{theorem} 
\label{main4}
For any $\alpha\ge 1$ the space $\fa^p$ contains nontrivial
infinite-dimensional nearly invariant subspaces.
\end{theorem} 

We will give several methods to construct nontrivial nearly invariant subspaces in
Fock-type spaces. In the classical Fock space $\mathcal{F}_2$ one can
give such examples using the Bargmann transform. In general spaces
$\fa$, $\alpha\ge 1$, one can define nearly invariant subspaces
imposing certain growth/decay conditions in some angles (see Section \ref{barg}). 

There exists a standard way to construct nearly invariant subspaces in a
Banach space $\mathcal{H}$ of entire functions with bounded point evaluations
and division property. 
Let $G \in \mathcal{H}$ have only simple zeros. Then it is easy to see that
\[
\mathcal{M}_G = \overline{{\rm Span}} \bigg\{ \frac{G(z)}{z-\lambda}: \
\lambda\in \mathcal{Z}_G \bigg\}
\]
is a nearly invariant subspace. Here $\mathcal{Z}_G$ is the zero set of $G$.
In the case when $G$ has multiple zeros, an obvious 
modification is required. Note that Corollary \ref{cor1} admits the
following reformulation in terms of approximation theory:
\begin{corollary}
\label{cor2}
If $\fw^p$ is a radial Fock-type space of zero exponential type, then 
the family $\big\{ \frac{G(z)}{z-\lambda}: \ \lambda\in \mathcal{Z}_G \big\}$
is complete in $\fw^p$ for any transcendental entire function  $G\in \fw^p$
with simple zeros.
\end{corollary} 

An interesting aspect related to these examples is exploited in
Section \ref{count}, where we construct a function $G$ with an asymmetric
behaviour which will be then  inherited by all the non-zero elements of the
corresponding subspace $\MM_G$.  In fact, all examples in Sections
\ref{barg} and \ref{count} are essentially based on finding  subspaces
whose non-zero elements share a certain asymmetric behaviour. This
leads to the natural question whether there exist   nontrivial
infinite-dimensional  nearly invariant subspaces satisfying additional
symmetry conditions. In Section \ref{rot} we consider  rotation invariance
of such subspaces  in radial Fock-type spaces, that is nearly invariant
subspaces which are also invariant for the operator
$R_\beta f(z)  = f(\e^{i\beta} z)$. 

It is easy to show that if $\beta/\pi \notin \mathbb{Q}$, then 
any nontrivial nearly invariant subspace $\MM$ 
of a radial space $\fw$ with  $R_\beta \MM \subset \MM$ must contain
polynomials up to some order and,  thus, is of the form 
$\pp_k$ (see Section \ref{rot}). 
 
The case when $\beta = \pi m/n$ with relatively prime $m$ and $n$ reduces 
to the case when $\beta = 2\pi /n$, and  we have the following result. 

\begin{theorem}
\label{simetr}
Let $\fw^p$ be a radial Fock-type space and let $n\ge 2$ be an integer such
that any element of $\fw^p$ has zero type with respect to order $n$. 
Then any nontrivial nearly-invariant subspace $\MM$ invariant with
respect to $R_{2\pi /n}$ is of the form $\pp_m$.
\end{theorem}

The restriction on the order and type is sharp. It is easy to see, e.g.,
that nontrivial subspaces of the Fock space from Subsection \ref{barg1}
can be invariant under $f\mapsto f(-z)$, which corresponds to $n=2$. 

We finish this Introduction with one open problem.
\medskip
\\
{\bf Problem.} {\it Is it true that any nearly invariant subspace
without common zeros in a Fock-type space $\fa^p$ is 1-generated, that is,
of the form $\mathcal{M}_G$ for some $G\in \fa^p$?}
\medskip

In what follows we write $U(x)\lesssim V(x)$ if 
there is a constant $C$ such that $U(x)\leq CV(x)$ holds for all $x$ 
in the set in question. We write $U(x)\asymp V(x)$ if both $U(x)\lesssim V(x)$ and
$V(x)\lesssim U(x)$. The standard Landau notations
$\Ordo$ and $\ordo$ also will be used.
The zero set of an entire function $f$ will be denoted  by $\mathcal{Z}_f$. 
We denote by $D(z,R)$ the disc with center $z$ of radius $R$ and by 
$m_2$  area-Lebesgue measure in $\CC$.
\bigskip
\\
\textbf{Acknowledgement.} The authors are grateful to
Eskil Rydhe for the suggestion to look at the problem in the context of $L^p$-norms. 
\bigskip


\section{Preliminaries on Cauchy transforms and proof of Theorem \ref{main1}}

In what follows we will use the following two results from \cite{bbb-fock}
about the asymptotic behaviour of Cauchy transforms
of measures in the plane.  More subtle results about  Cauchy transforms on
rectifiable curves were obtained in \cite{mm, verd}.  We say that
$\Omega\subset \CC$ is a {\it set of zero planar density} if 
\[
\lim_{R\to+\infty} \frac{m_2(\Omega \cap D(0, R))}{R^2} = 0.
\]
It is well-known that for any finite complex Borel measure $\nu$ in the plane,
its Cauchy transform
$$
\mathcal{C}_{\nu}(z) = \int_\mathbb{C}\frac{d\nu(\xi)}{z-\xi} 
$$
is well-defined in the principal value sense $m_2$-a.e. Note that if $\nu$ is
compactly supported, we obviously have
\begin{equation}
\label{plan}
\mathcal{C}_{\nu}(z) = \frac{\nu(\mathbb{C})}{z} + \ordo\Big(\frac{1}{z}\Big),
\qquad |z|\to +\infty.
\end{equation}
From the results of \cite{mm} one can easily deduce that the same
asymptotics holds for an arbitrary $\nu$ outside some ``small'' set
(see \cite[Proof of Lemma 4.3]{bbb-fock}):
\[
\mathcal{C}_{\nu}(z) = \frac{\nu(\mathbb{C})}{z} + \ordo\Big(\frac{1}{z}\Big),
\qquad |z|\to +\infty, \ \ 
z\in \CC\setminus\Omega,
\]
where $\Omega$ is a set of zero planar density.

The following result of A. Borichev (see \cite[Lemma 4.2]{bbb-fock} and
also \cite{bbb-arxiv} where an inaccuracy of the proof is corrected) 
can be considered as an extension of the classical Liouville theorem. 

\begin{theorem}
\label{dens}
If an entire function $f$ of finite order is bounded on 
$\CC\setminus \Omega$ for some set $\Omega$ of zero planar density, 
then $f$ is a constant. 
\end{theorem}

We will also need the following simple observation:
for any finite complex measure $\nu$ we have 
\begin{equation}
\label{ct} 
\int_{\co}\frac{|\mathcal{C}_{\nu}(z)|}{1+|z|^3}dm_2(z)\le 10 |\nu|(\co).
\end{equation}
This  follows directly from Fubini's theorem and the estimate
\[
\int_\co\frac1{|z-\zeta|(1+|z|^3)}dm_2(z)\le \int_{|z-\zeta|\ge 1}
\frac1{1+|z|^3}dm_2(z)  +
\int_{|z-\zeta|< 1}\frac1{|z-\zeta|}dm_2(z),
\]
valid for any $\zeta\in \co$.

Our proofs are often based on duality arguments. We cannot identify the
dual space to $\fw^p$ with the space $\fw^q$, $1/p+1/q=1$, unless $p=2$.
However, by the Hahn--Banach theorem, any continuous linear functional
on $\fw^p$, $p\in [1, +\infty)$, may be represented in the form 
\begin{equation}
\label{Hahn-Banach} 
\Psi_g(f)=\int_\co fg W dm_2,\qquad g\in L^q(W),
\end{equation}
where $1/p+1/q=1$, $L^q(W) = \{f: f W^{1/q} \in L^q(\mathbb{C}, dm_2)\}$.
In the proof of Theorem \ref{main11} we will need a more detailed
information about the structure of the functionals and will show that
the function $g$ can be chosen to have
a certain ``antianalyticity'' (see Proposition \ref{anal}). 

\begin{proof}[Proof of Theorem \ref{main1}]
Let $\MM$ be a nontrivial backward shift invariant subspace of $\fw^p$ and let 
$\Psi_g$, $g\in L^q(W)$, be a functional annihilating $\MM$. 
To simplify notation we write $d\mu$ in place of $W m_2$.
Let $f \in \MM$. Then we have, for any $z\in \co$, 
$$
\int_\co \frac{f(\zeta) - f(z)}{\zeta-z} g(\zeta) d\mu(\zeta) = 0, 
$$
whence
$$
f(z) \mathcal{C}_{g\mu}(z) = \mathcal{C}_{fg\mu}(z). 
$$ 
Since $\pp$ is dense in $\fw^p$, we can choose the smallest $n\in \zl_+$ such that 
$\Psi_g(z^n) = \int_\co \zeta^n g(\zeta) d\mu(\zeta) \ne 0$.
Then 
$$\int_\co \frac{z^n-\zeta^n}{z-\zeta} g(\zeta) d\mu(\zeta) = 0 
$$ 
whence, by \eqref{plan}, 
there exists a set $\Omega$ of zero planar density such that
$$
\mathcal{C}_{g\mu}(z) = 
\frac{1}{z^n} \mathcal{C}_{\zeta^n g\mu}(z)
=
\frac{\alpha}{z^{n+1}} + \ordo\Big(\frac{1}{z^{n+1}}\Big), \qquad 
\mathcal{C}_{fg\mu}(z) = \Ordo\Big(\frac{1}{z}\Big),
$$
when $|z|\to+\infty$, $z\in \CC\setminus\Omega$, and $\alpha \ne 0$. 
We conclude that $|f(z)| = \Ordo(|z|^n)$ outside a set of zero planar measure.
Since $\fw^p$ is of finite order, $f$ is a polynomial of degree at most $n$ 
by Theorem \ref{dens}.
\end{proof}

\begin{remark} 
{\rm 1. The above argument applies to a somewhat more general class of
spaces than Fock-type spaces. The only property of the space we used is that
there exists a measure $\mu$ in $\co$ such that 
$\|f\|_{\fw^p} = \|f\|_{L^p(\mu)}$ for any $f\in \fw^p$. 
\smallskip
\\
2. In general,  if the conditions 
of Theorem \ref{main1} are not satisfied, even the ordering property for
backward shift invariant spaces may fail. Simple examples of this type are
provided by Paley--Wiener spaces. Let us denote, as usual, by $PW_I$  the
image  of $L^2( I)$ by the the Fourier transform. 
As mentioned in the Introduction, by the result in \cite{bom}, $PW_{[-a,a]}$
is a Fock-type space, and $PW_{[0,a]}$, $PW_{[-a, 0]}$
are backward shift invariant  subspaces which are orthogonal to each other. }
\end{remark}
\bigskip


\section{Proof of Theorem \ref{main11}}

The density of polynomials in radial spaces $\fw^p$ 
is apparently well known. We include a very short
proof of this fact for the sake of completeness. 

\begin{proposition}
\label{den} 
If $W$ is a radial weight with the property $\int_{\co}|z|^n W(z) dm_2(z)<+\infty$
for any $n$, then polynomials are dense in $\fw^p$ for all $p\in [1, +\infty)$.
\end{proposition}

\begin{proof} 
Given $f\in \fw^p$ and $\theta\in [0,2\pi]$, we have that  
$f_\theta(z)=f(\e^{\imag\theta}z)$ belongs to $\fw^p$ and 
$\|f_\theta\|_{\fw^p}=\|f\|_{\fw^p}$. In particular, if  $\theta_n\to \theta$, 
from the equality of the norms  together with the  pointwise convergence 
$f_{\theta_n}(z)\to f_\theta(z)$, $z\in \co$, we deduce that
$\|f_{\theta_n}-f_\theta\|_{\fw^p}\to 0$, 
Now let $g\in L^q(W)$ or ($L^\infty(\co)$ if $p=1$) 
annihilate the polynomials, let  $f\in \fw^p$ and let 
$$
u(\e^{\imag\theta})=\int_\co f_\theta g W dm_2, \qquad \theta\in [0,2\pi].
$$
By the above argument $u$ is continuous on the unit circle. By Fubini's theorem, 
its Fourier coefficients satisfy
\[
\int_0^{2\pi}u(\e^{\imag\theta})\e^{-\imag n\theta}\frac{d\theta}{2\pi}
=\int_\co  g(z)
\Big( \int_0^{2\pi}f_\theta(z)\e^{-\imag n\theta}\frac{d\theta}{2\pi}\Big) W(z)
dm_2(z)=0, \quad  n\in\zl,
\]
since the inner integral either vanishes, or it is a constant multiple of $z^n$. 
Thus $u=0$ on the unit circle, whence $g$ annihilates any $f\in \fw^p$. 
\end{proof}

A key step in the proof of Theorem \ref{main11} is a simple identity  
for continuous linear functionals $\Psi_g$ on $\fw^p$.  With the representation 
\eqref{Hahn-Banach}  we have that if $f\in \fw^p$, $\zeta\in \co$, and 
$T_\zeta f(z)=L(1-\zeta L)^{-1}f(z)=\frac{f(z)-f(\zeta)}{z-\zeta}$, then 
\begin{equation}
\label{basic-identity}
\Psi_g(T_\zeta f)=f(\zeta)\mathcal{C}_{g\mu}(\zeta)-\mathcal{C}_{fg \mu}(\zeta),
\end{equation}
$m_2$-a.e., where $\mu=W m_2$. In order to prove Theorem \ref{main11}, 
we shall assume that the left hand side vanishes and estimate  
$|\mathcal{C}_{fg\mu}|$ from above and $|\mathcal{C}_{g\mu}|$ from below,  
in order to obtain an estimate for $|f|$. While \eqref{ct} provides 
a sufficiently good upper estimate for $\mathcal{C}_{fg\mu}$, the lower estimate for 
$\mathcal{C}_{g\mu}$ is more subtle and is addressed in the proposition  below.

We shall denote by $H^p$ the usual Hardy spaces on the unit disc $\mathbb{D}$,
by $A$ the disc algebra and by $P_+$ the usual Riesz projection
\[
P_+u(\rho \e^{\imag\theta})=\frac1{2\pi}\int_0^{2\pi}
\frac{u(\e^{\imag t})}{1-\rho \e^{\imag(\theta-t)}}dt, 
\qquad \rho<1,\, \theta\in [0,2\pi],
\]
where $u$ is a function in $L^1$ on the unit circle.
Recall that functions in  $H^p$  (or $A$) can be identified  with Fourier series in 
$L^p([0,2\pi])$ (or $C([0,2\pi])$) whose negative coefficients vanihsh, and this 
identification defines an isometry between the spaces in question.
Moreover, by the famous M. Riesz theorem $P_+$ is  a bounded operator from
$L^p([0,2\pi])$ onto $H^p$, whenever $1<p<+\infty$.

\begin{proposition}
\label{anal} 
Let  $1<p<+\infty$, $\frac1{p} +\frac{1}{q} =1$, and let $\fw^p$
contain all polynomials. 
Then there exists $K>0$ depending only on $p$ such that
every continuous linear functional $\Psi$ on $\fw^p$ can be represented
in the form \eqref{Hahn-Banach}	with $g\in L^q(W)$ such that\textup:
\begin{itemize}
\item[(i)]
$\|g\|_{L^q(W)} \le K\|\Psi\|$, $g\in C(\co)$,  and for all 
  $r>0$, there exists a function $u_r \in A$ such that
  $u_r(\e^{\imag t})=\overline{g}(r\e^{\imag t})$\textup;
\item[(ii)]
The Cauchy transform $\mathcal{C}_{g\mu}$ is continuous on $\co$,
satisfies for $r>0$
\begin{equation}
\label{bnm1}
\mathcal{C}_{g\mu}(r\e^{\imag t})=\frac{2\pi}{r\e^{\imag t}}\int_0^r
\overline{u_\rho}\Big(\frac{\rho \e^{\imag t}}{r}\Big)
\rho W(\rho) d\rho;
\end{equation}
\item[(iii)]
There exists a function $U_r \in A$ such that
$U_r(\e^{\imag t})=\overline{\mathcal{C}_{g\mu}}(r\e^{\imag t})$
and 
\begin{equation}
\label{bnm2}
\|U_r\|_{H^q} \le \frac{K}{r}\|\Psi\| \Big(
\int_0^r\rho W(\rho) d\rho\Big)^{\frac1{p}}.
\end{equation}
\end{itemize}
\end{proposition}

\begin{proof} (i) Let $\Psi$ be a continuous linear functional on $\fw^p$ 
and choose, by the Hahn-Banach theorem, $h\in L^q(W)$ such that
$$ 
   \Psi(f)=\int_\co fh W dm_2, \qquad \|h\|_{L^q(W)}=\|\Psi\|.
$$
Since $1<p<+\infty$,  $\fw^p$ is reflexive, since it is  a  closed subspace
of the reflexive space $L^p(W)$. Therefore,  
there exists $f_0\in \fw^p$, $\|f_0\|_{\fw^p}=1$, 
such that $\Psi(f_0)=\|\Psi\|$. Then for the functions $f_0 W^{\frac1{p}}$, 
$hW^{\frac1{q}}$, we have equality in the H\"older inequality, which implies 
that $h=\overline{f}_0|f_0|^{p-2}$, $m_2$-a.e. on $\co$. 
Put $F_{r}(\e^{\imag t})=\overline{f}_0|f_0|^{p-2}(r\e^{\imag t})$,
$r\ge 0$, $t\in [0,2\pi]$.
It is easy to verify that $\overline{f}_0|f_0|^{p-2}\in \Lip^\beta(\Omega)$
for any compact subset 
$\Omega$ of $\co$ with $\beta = \min\{p,2\}-1$. Therefore, the function 
$u_r = P_+ F_r$ belongs to $A$ and, moreover, is 
in $\Lip^\beta (\overline{\mathbb{D}})$ (see, e.g., \cite[Theorem 5.1]{dur}). Let
\[
  g(r\e^{\imag t})=\overline{u_r} (\e^{\imag t}). 
\]
Then the function $g$ is well-defined everywhere in $\co$. Let us show that
$g$ is continuous.
Let
$$
g_n(r\e^{\imag t})=\overline{u_r}\Big(\big(1-\frac1{n}\big)\e^{\imag t}\Big).
$$
Since the functions $u_r$  are in $\Lip^\beta (\overline{\mathbb{D}})$ 
uniformly with respect to $r$ in any bounded set,
we conclude that $g_n(r\e^{\imag t})$ converge to $g(r\e^{\imag t})$
locally uniformly in $r$ and uniformly in $t$.
Finally, it is obvious that $g_n (r'\e^{\imag t})$ converges uniformly in $t$ to
$g(r\e^{\imag t})$ as $r'\to r$  for any fixed $n$. We conclude that  $g\in C(\co)$.

Using polar coordinates the estimate
$\|g\|_{L^q(W)} \le K\|h\|_{L^q(W)} = K\|\Psi\|$ 
follows by the M.~Riesz theorem. Finally, since $\fw^p$ consists
of functions which are continuous on every closed disc and analytic
inside, we have by definition 
\[ \int_0^{2\pi} 
f(r\e^{\imag t}) (u_r(\e^{\imag t}) - F_r (\e^{\imag t}))dt=0,
\]
hence using polar coordinates, 
\begin{align*}
\Psi(f) &= \int_{0}^{+\infty} rW(r) \int_0^{2\pi}  f(r \e^{\imag t} )
F_r (\e^{\imag t}) dt dr
\\&
= \int_{0}^{+\infty} rW(r) \int_0^{2\pi} 
f(r \e^{\imag t}) u_r(\e^{\imag t}) dt dr =\int_\co fg W dm_2.
\end{align*}
(ii) The function $\mathcal{C}_{g\mu}$ is well-defined everywhere in $\co$
and, for a fixed $\zeta \ne 0$, we have
\[
\mathcal{C}_{g\mu}(\zeta) = 
\lim_{\delta\to 0}\int_{\big| |z|-|\zeta| \big|>\delta}\frac{g(z)}{\zeta-z}
W(z)dm_2(z).
\]
Using (i) and a straightforward calculation we obtain
\[
\int_0^{2\pi}\frac{g(\rho \e^{\imag t})}{\zeta-\rho \e^{\imag t}}dt=0, \
\text{if}\ \rho>|\zeta|, \qquad 
\int_0^{2\pi}\frac{g(\rho \e^{\imag t})}{\zeta-\rho \e^{\imag t}}dt=
\frac{2\pi}{\zeta} \overline{u}_\rho\big({\rho}/{\overline{\zeta}}\big),
\ \text{if}\ \rho <|\zeta|.
\]
Hence,
\[
\mathcal{C}_{g\mu}(\zeta)= \frac{2\pi}{\zeta}\lim_{\delta\to 0}
\int_0^{|\zeta|-\delta}
\overline{u}_\rho\Big(\frac{\rho}{\overline{\zeta}}\Big) \rho W(\rho) d\rho
=
\frac{2\pi}{\zeta}\int_0^{|\zeta|}
\overline{u}_\rho\big({\rho}/{\overline{\zeta}}\big) \rho W(\rho)d\rho,
\]
since the functions $u_\rho$ are uniformly bounded for $\rho\le |\zeta|$. 
This proves formula \eqref{bnm1}.
In particular, $\mathcal{C}_{g\mu}$ is continuous at $\zeta$.

(iii) It follows from \eqref{bnm1} that the nonnegative Fourier
coefficients of the function $\mathcal{C}_{g\mu}(r\e^{\imag t})$ vanish, therefore 
there exists a function 
$U_r \in A$ such that $U_r(\e^{\imag t})
=\overline{\mathcal{C}_{g\mu}}(r\e^{\imag t})$,
$U_r(0) = 0$.
Finally, by (i),
$\|u_r\|_{H^q}^q = \frac1{2\pi} \int_0^{2\pi}|g(r\e^{\imag t})|^qdt$,
and the inequality \eqref{bnm2} follows by the H\"older inequality together
with the fact that
\[
\int_0^{2\pi}|u_\rho(s\e^{\imag t})|^qdt\le 2\pi \|u_\rho\|_{H^q}^q,
\qquad 0< s\le 1.
\]
\end{proof}

\begin{proof}[Proof of Theorem \ref{main11}] 
The density of polynomials in $\fw^p$ has already  
been  proved in Proposition \ref{den}. We are going 
to show that if $f\in \fw^p$ and $\Psi$ is a non-zero continuous linear functional 
on $\fw^p$ with $\Psi(T_\zeta f)=0$, $\zeta\in \co$, then $f$ is a polynomial of 
degree at most $n+1$, where $n$ is the smallest nonnegative integer such 
that $\Psi(z^n)\ne 0$. Since polynomials are dense in $\fw^p$ and $\Psi \ne 0$, 
such an integer exists.

We shall consider first the case when $n=0$. Consider the representation of
$\Psi$ given by Proposition \ref{anal}. The equality \eqref{basic-identity} gives
\begin{equation}
\label{basic-identity1} 
f(\zeta)\mathcal{C}_{g\mu}(\zeta) = \mathcal{C}_{fg\mu}(\zeta),
\end{equation}
$m_2$-a.e. 
For any fixed $r>0$, the function $\e^{\imag t}\mapsto 
f(r\e^{\imag t})U_r(\e^{\imag t})$ belongs to the disc algebra $A$. Using also
\eqref{basic-identity1} we obtain for all $\delta\in (0,1)$ 
\begin{equation}
\label{h1-est}
\begin{aligned} 
 \int_0^{2\pi}|f(\delta r\e^{\imag t})U_r(\delta \e^{\imag t})|dt& \le 
 \int_0^{2\pi}|f(r\e^{\imag t})U_r(\e^{\imag t})|dt \\ =
 & \int_0^{2\pi}|f(r\e^{\imag t})\mathcal{C}_{g\mu}(r\e^{\imag t})|dt 
 = \nonumber \int_0^{2\pi}|\mathcal{C}_{fg\mu}(r\e^{\imag t})|dt.
\end{aligned}
\end{equation}

Given an $H^q$-function $v$ with $v(0)=0$, we can write for $|z|<1$, 
$v(z)=zv'(0)+z^2v_1(z)$, where $\|v_1\|_{H^q}\le  2\|v\|_{H^q}$, and 
a standard estimate gives for all $\delta\in (0,1/2)$,
\[
|v(\delta \e^{\imag t})|\ge \delta|v'(0)|-
\frac{2^{1+\frac1{q}}\delta^2}{(1-\delta)^{\frac1{q}}}\|v\|_{H^q}
\ge \delta|v'(0)|- K'\delta^2\|v\|_{H^q}
\]
with $K'$ depending on $q$ only.

In order to apply this estimate to $U_r$ we note first that 
\begin{align*} 
U_r'(0) & =\frac{2\pi}{r}\int_0^r u_\rho(0)\rho W(\rho)d\rho
\\
& = \frac{1}{r}\int_0^r \Big( \int_0^{2\pi}u_\rho(\e^{\imag t}))dt\Big)
\rho W(\rho)d\rho 
= \frac{1}{r}\int_{|z|<r}\overline{g}(z)W(z)dm_2(z),
\end{align*}
whence $r U_r'(0) \to \overline{\Psi (1)}$ as $r\to+\infty$.
Together with the estimate \eqref{bnm2} this gives for $r>0$, $\delta\in (0,1/2)$,
\[
|U_r(\delta \e^{\imag t})| \ge \frac{\delta}{r} \left|\int_{|z|<r}\overline{g}(z)
W (z)dm_2(z)\right|- K''\delta^2 \|\Psi\|,
\]
with $K''=K' K\big(\int_\co W (z)dm_2(z)\big)^{\frac1{p}}$. Then there exists 
$r_0>0$, $\delta_0\in (0,1)$, such that for $r>r_0$, $\delta\in (0,\delta_0)$ 
we have
\[
|U_r(\delta \e^{\imag t})| \ge \frac{\delta|\Psi(1)|}{2r}.
\]
Using this in \eqref{h1-est} we obtain for $r,\delta$ as above,
\[ 
\int_0^{2\pi}|f(\delta r\e^{\imag t})|dt \le  \frac{2r}{\delta|\Psi(1)|}
\int_0^{2\pi}|\mathcal{C}_{fg\mu}(r\e^{\imag t})|dt,
\]
and from \eqref{ct} it follows that 
\[
\int_\co\frac{|f(\delta z)|}{1+|z|^4}dm_2(z)<+\infty.
\]
This implies that $f$ is a polynomial of degree at most one and 
finishes the proof in the case when $\Psi(1)\ne 0$.

If the smallest integer $n$ with $\Psi(z^n)\ne 0$ is strictly  positive,
we observe that $\Psi_n (f)=\Psi (z^nf)$ defines a continuous linear
functional on the space $\mathcal{F}_{W_n}^p$ with radial weight
$W_n(z)=|z|^{-n}W(z)$. The space $\mathcal{F}_{W_n}^p$  contains polynomials
and all functions of the form $L^n h$, $h\in \fw^p$. 
Moreover, since $h-z^n L^n h$ is a polynomial of degree strictly less than $n$, 
and $L^n T_\zeta=T_\zeta L^n$,  we have 
\[
\Psi_n(T_\zeta L^n f)= \Psi_n(L^n T_\zeta f)=\Psi(z^n  L^n T_\zeta f)=
\Psi(T_\zeta f)=0,\qquad \zeta\in \co.
\]
Since $\Psi_n(1)\ne 0$, the previous argument shows that $L^n f$ is a
polynomial of degree at most one, which completes the proof. 
\end{proof}
\bigskip


\section{The Ordering Theorem for  Fock-type spaces of small growth}
\label{orde}

The key idea of the proof of Theorem \ref{main3} is due to
L.~de~Branges \cite[Theorem 35]{br}. 
It is based on an application of a deep result by M. Heins \cite{heins}
(see also \cite[Lemma 8]{br} or \cite[Proposition 26]{rom}): {\it if 
$f$ and $g$ are entire function of zero exponential type and
$$
\min\big(|f(z)|, |g(z)|\big) \lesssim 1, \qquad z\in\co,
$$
then either $f$ or $g$ is a constant function.} However, we will need a
slightly refined version of Heins' theorem where the estimate holds
everywhere up to a small exceptional set. 
The proof of the following statement is identical to the proof of Heins' theorem 
(see proof of Proposition 26 in \cite{rom}) and we omit it. 

\begin{theorem}
\label{hei}
Let $f$ and $g$ be entire function of zero exponential type such that
$$
\min\big(|f(z)|, |g(z)|\big) \lesssim 1, \qquad z\in\co\setminus\Omega,
$$
where the set $\Omega$ satisfies
$$
\int_\Omega \frac{dm_2(z)}{|z|+1}<+\infty. 
$$
Then either $f$ or $g$ is a constant function.
\end{theorem}

\begin{proof}[Proof of Theorem \ref{main3}]
Let $\MM_1$, $\MM_2$ be two nearly invariant subspaces of $\fw^p$.  
Assume that neither $\MM_1 \subset \MM_2$ nor
$\MM_2 \subset \MM_1$. Then we can choose two nonzero functionals
$\Psi_{F_1}$ and $\Psi_{F_2}$, $F_1, F_2 \in L^q(W)$ such that 
$\Psi_{F_1}$ annihilates $\MM_2$ but not $\MM_1$, while
$\Psi_{F_2}$ annihilates $\MM_1$ but not $\MM_2$.

Let $F\in \MM_1$ and $G\in \MM_2$. Define two functions
\begin{equation}
\label{ff}
f(z)  = \Psi_{F_1} \bigg(\frac{F - \frac{F(z)}{G(z)}G}{\zeta-z}\bigg)
=  \int_{\co} \frac{F(\zeta) - \frac{F(z)}{G(z)}G(\zeta)}{\zeta-z}
F_1(\zeta) d\mu(\zeta), 
\end{equation}
\begin{equation}
\label{fg}
g(z)  = \Psi_{F_2} \bigg( \frac{G - \frac{G(z)}{F(z)}F}{\zeta-z}\bigg) 
= \int_{\co} \frac{G(\zeta) - \frac{G(z)}{F(z)}F(\zeta)}{\zeta-z}
F_2(\zeta) d\mu(\zeta),
\end{equation}
where $\mu = W m_2$. The functions $f$ and $g$ are well-defined and analytic
on the sets $\{z: G(z)\ne 0\}$
and $\{z: F(z)\ne 0\}$, respectively.
\medskip
\\
{\bf Step 1:} {\it $f$ and $g$ are entire functions of zero exponential type, 
$f$ does not depend on the choice of $G$ and $g$ does not depend on the
choice of $F$}.

Let $f_1$ be a function associated in a similar way to $G_1 \in \MM_2$,
$$
f_1(z)  = 
\int\frac{F(\zeta) - \frac{F(z)}{G_1(z)}G_1(\zeta)}{\zeta-z} F_1(\zeta) d\mu(\zeta).
$$
Then, for $G(z) \ne 0$ and $G_1(z) \ne 0$, we have
$$
f_1(z) - f(z) = \frac{F(z)}{G(z)G_1(z)} 
\int \frac{G_1(z)G(\zeta) - G(z)G_1(\zeta)}{\zeta-z} F_1(\zeta) d\mu(\zeta) = 0,
$$
since $\frac{G_1(z)G - G(z)G_1}{\zeta-z} \in \MM_2$.

Now choosing $G$ such that $G(z) \ne 0$ we can extend $f$ analytically to 
a neighborhood of the point $z$. Thus, $f$ and $g$ are entire functions.

Since $fG = G\mathcal{C}_{FF_1\mu} - F \mathcal{C}_{GF_1\mu}$, it follows
from \eqref{ct} that 
$$
\int_{D(z, 1)} |f(\zeta)G(\zeta)|dm_2(\zeta)
\lesssim (|z|^3+1) \sup_{\zeta\in D(z, 1)} \big( |F(\zeta)| +|G(\zeta)|\big),
$$
whence $fG$ is an entire function of zero type. 
Since $G$ itself is of zero exponential type,
$f$ is of zero type by the standard properties of entire functions (see,
e.g., \cite[Chapter 1, \S 9]{lev}).
\medskip
\\
{\bf Step 2:} {\it Either $f$ or $g$ is identically zero.}

Given $z$ such that $F(z)\ne 0$, $G(z) \ne 0$, we have
\begin{equation}
\label{ots}
\begin{aligned}
|f(z)| & \le \bigg|\int \frac{F(\zeta)F_1(\zeta)}{\zeta-z} d\mu(\zeta) \bigg| + 
\bigg|\frac{F(z)}{G(z)}\bigg|\cdot 
\bigg|\int \frac{G(\zeta)F_1(\zeta)}{\zeta-z} d\mu(\zeta) \bigg|,\\
|g(z)| & \le \bigg|\int \frac{G(\zeta)F_2(\zeta)}{\zeta-z} d\mu(\zeta) \bigg| + 
\bigg|\frac{G(z)}{F(z)}\bigg|\cdot 
\bigg|\int \frac{F(\zeta)F_2(\zeta)}{\zeta-z} d\mu(\zeta)\bigg|.
\end{aligned}
\end{equation}
Note that, by \eqref{ct}, for any finite measure $\nu$, there exists a set
of finite measure such that
$|C_\nu(z)| \le |z|^3$ outside this set. Then we can find a set $\Omega$ of
finite measure such that
$$         
|f(z)| \lesssim |z|^3\bigg(1+ \bigg|\frac{F(z)}{G(z)}\bigg|\bigg), \qquad
|g(z)| \lesssim |z|^3\bigg(1+ \bigg|\frac{G(z)}{F(z)}\bigg|\bigg),
\qquad z\notin \Omega.
$$
We conclude that
$$
\min\big(|f(z)|, |g(z)|\big) \lesssim |z|^3, \qquad z\notin \Omega.
$$
By Theorem \ref{hei} either $f$ or $g$ is a polynomial. 

Assume that $f$ is a nonzero polynomial. By \eqref{plan}, 
there exists a set $\Omega$ of zero area density such that
$$
\bigg|\int\frac{F(\zeta) F_1(\zeta)}{\zeta-z} d\mu(\zeta) \bigg|
+
\bigg|\int\frac{G(\zeta) F_1(\zeta)}{\zeta-z} d\mu(\zeta) \bigg| 
=\Ordo\Big(\frac{1}{|z|}\Big), \qquad z\notin \Omega. 
$$
Hence, $|F(z)/G(z)| \to+\infty$ as $|z|\to+\infty$, $z\notin\Omega$, and so
$$
|g(z)| \le \bigg|\int\frac{G(\zeta) F_2(\zeta)}{\zeta-z} d\mu(\zeta) \bigg| + 
\bigg|\frac{G(z)}{F(z)}\bigg|\cdot 
\bigg|\int \frac{F(\zeta) F_2(\zeta)}{\zeta-z} d\mu(\zeta)\bigg| 
= \Ordo\Big(\frac{1}{|z|}\Big), \qquad w\notin \Omega\cup \widetilde{\Omega}, 
$$
where $\widetilde{\Omega}$ is another set of zero area density.
Thus, $g$ tends to zero outside a set of zero density and so $g\equiv 0$ by
Theorem \ref{dens}.
\medskip
\\
{\bf Step 3:} {\it End of the proof.} 

Without loss of generality, let 
$f\equiv 0$. Then
$$
\frac{F(z)}{G(z)}\int\frac{G(\zeta) F_1(\zeta)}{\zeta-z}  d\mu(\zeta)
=\int\frac{F(\zeta) F_1(\zeta)}{\zeta-z}  d\mu(\zeta)
$$                                             
for any $F\in \MM_1$, $G\in \MM_2$.

Recall that $\Psi_{F_1}$ does not annihilate $\MM_1$ and so we can choose
$F \in \MM_1$ such that
$\Psi_{F_1}(F) = \int F F_1 d\mu \ne 0$. Then, by \eqref{plan},
$$
\bigg|\int\frac{F(\zeta)F_1(\zeta)}{\zeta-z}  d\mu(\zeta)\bigg|
\gtrsim \frac{1}{|z|}, 
\qquad z\notin \Omega,
$$
for some set $\Omega$ of zero density. Since $\Psi_{F_1}(G) =0$ for 
any $G\in \MM_2$, 
we have (again by \eqref{plan}) 
$$
\bigg|\int\frac{G(\zeta)F_1(\zeta)}{\zeta-z}  d\mu(\zeta)\bigg| =
\ordo\Big(\frac{1}{|z|}\Big), 
\qquad |z|\to+\infty, \ z\notin \widetilde{\Omega},
$$
where $\widetilde{\Omega}$ is another set of zero density. We conclude that
$|F(z)/G(z)| \to +\infty$ when $|z|\to +\infty$ 
outside the set of zero density 
$\Omega \cup \widetilde{\Omega}$ (for any $G\in\MM_2$). 
Applying this fact and \eqref{plan} to $g$ we conclude that $|g(z)|\to 0$ 
outside a set of zero density and so $g\equiv 0$ by Theorem \ref{dens}.

Thus, for any $G\in\MM_2$, we have
\begin{equation}
\label{ghj}
\frac{G(z)}{F(z)}\int\frac{F(\zeta)F_2(\zeta)}{\zeta-z}  d\mu(\zeta)
=\int\frac{G(\zeta) F_2(\zeta)}{\zeta-z}  d\mu(\zeta)
\end{equation}
and we may repeat the above argument. Choose $G \in \MM_2$
such that $\Psi_{F_2}(G) = \int G F_2 d\mu \ne 0$. Then,
by \eqref{plan}, the modulus of the 
right-hand side in \eqref{ghj} is $\gtrsim |z|^{-1}$,
while the left-hand side is $o(|z|^{-1})$ when $|z|\to+\infty$ 
outside a set of zero density. This contradiction proves Theorem \ref{main3}.
\end{proof}
\bigskip


\section{Failure of  the Ordering Theorem in spaces of large growth}
\label{barg}

\subsection{Examples via the Bargmann transform} 
\label{barg1}

We start with the case of the classical space $\mathcal{F}_2$ with the
weight $W(z) = \exp(-|z|^2)$.
The Bargmann (Segal--Bargmann) transform is the isometry $\mathcal{B}$ 
from $L^2(\mathbb{R})$ onto $\mathcal{F}_2$ given by 
$$
\mathcal{B} u(z)=\int_\mathbb{R}\e^{-\frac{t^2+z^2}{2}+\sqrt{2}tz} u(t)dt,
\qquad u\in L^2(\mathbb{R}).
$$
Let $I\subsetneq  \mathbb{R}$ be a nonempty open interval and let 
$$
\mathcal{M}_I=\mathcal{B} L^2_I, \qquad L^2_I = \big\{u:~u=0 \ \text{a.e. off }
I\big\}.
$$ 
Obviously, $\mathcal{M}_I$ is nontrivial and closed in $\mathcal{F}_2$.  
We want to show that these subspaces are nearly invariant. This will follow 
directly from an alternative description in terms of the
Paley--Wiener type spaces $PW_I$.
Recall that $PW_I$ is by definition the image of $L^2_I$ under the Fourier
transform.  
Also, if $I=[a,b]$, then $f\in PW_I$ if and only if $f$ is the restriction to
the real axis of an entire function $F$ which is in $L^2(\rl)$ and satisfies
$$
|F(z)| \le 
\begin{cases}
C\e^{b\ima z}, & \ima z>0,\\
C\e^{a\ima z}, & \ima z<0.
\end{cases}
$$

\begin{proposition}\label{ex-via-Bargmann} 
For a nonempty bounded open interval  $I\subset \mathbb{R}$ we have
$$
\mathcal{M}_I=\big\{f\in \mathcal{F}_2: \ \e^{-\frac{z^2}{2}}f(-\imag z) \in
PW_{\sqrt{2}I}\big\}.
$$  
In particular, $\mathcal{M}_I$ is nearly invariant with no common zeros.
\end{proposition}

\begin{proof}   From the definition of $\mathcal{B}$ we have for
$f=\mathcal{B} u \in \mathcal{F}_2$, $u\in L^2_I$,  
$$
\e^{-\frac{z^2}{2}}f(-\imag z)=\int_\mathbb{R}
\e^{-\frac{t^2}{2}-\sqrt{2}\imag zt} u(t)dt
\in PW_{\sqrt{2} I}.
$$
Conversely, any function $v$ in $L^2_I$  with bounded $I$ can be written as
$v(t) = \e^{-t^2/2}u(t)$ and so 
$$
\int_I \e^{-\sqrt{2}\imag zt} v(t)dt = e^{-\frac{z^2}{2}}f(-\imag z)
$$
for $f =\mathcal{B} u\in \mathcal{F}_2$.

It follows immediately from the structure of the Paley--Wiener spaces
that $\mathcal{M}_I$ is nearly invariant. 
\end{proof}

Clearly, subspaces $\mathcal{M}_I$ are not ordered by inclusion. 
For example, if  $I\cap J=\emptyset$, we have
$\mathcal{M}_I\cap\mathcal{M}_J=\{0\}$, and both are nontrivial. 


\subsection{Phragm\'en--Lindel\"of approach}

Next we consider a more general construction of nontrivial 
nearly invariant subspaces in the radial Fock-type spaces $\fa^p$,
$\alpha\ge 1$, with the weight $W_\alpha(z) = \exp(-|z|^\alpha)$. 
The idea behind our construction is to consider subspaces consisting
of functions satisfying a growth  restriction in a fixed angle. The
remarkable fact is that the growth restriction and the angle can be chosen
such that the  Phragm\'en--Lindel\"of principle guarantees that the
corresponding subspace is closed in $\fa^p$.

Let us record first the standard pointwise estimate in these spaces.  To this
end, note that for any $\alpha\ge 1$ we have $W_\alpha(z)\asymp W_\alpha(\zeta)$,
$|\zeta-z| \le |z|^{2-2\alpha}$, 
$|z|\ge 1$. Thus, using the fact that
$|f(z)|^p \le \frac{1}{\pi r^2} \int_{D(z, r)} |f(\zeta)|^p dm_2(\zeta)$,
we get 
\begin{equation}
\label{mir}
|f(z)|^p \e^{-|z|^\alpha} \le A |z|^{2\alpha-2} \|f\|^p_{\fa^p}, \qquad |z|>1,
\end{equation}
for some constant $A = A(\alpha, p)$. 

Now let $\Omega = \{z: 0< \arg z <  2\pi/\alpha\}$ and let $\Gamma =
\partial \Omega$. Fix some polynomial $P$ with zeros outside $\overline{\Omega}$ and ${\rm deg}\,
P \ge (2\alpha-2)/p$. For any $a\ge 0$, put 
$$
\mathcal{M}_a = \bigg\{ f\in \fa^p:\ \Big| \frac{f(z)}{P(z)}
\e^{-\frac{1}{p} z^\alpha} \Big| 
\le C \e^{a|z|^{\alpha/2}}, \ z\in\Omega,  \ \text{for some} \ C>0 \bigg\}.
$$

\begin{proposition}
\label{phl} 
$\mathcal{M}_a$ is a closed nearly invariant subspace of $\fa^p$, with no
common zeros in $\mathbb{C}$.
\end{proposition}

\begin{proof} 
It is sufficient to show that $\mathcal{M}_a$ is closed, since the remaining
assertions are immediate.
Let $f_n\in \mathcal{M}_a$ be a sequence convergent in $\fa^p$ to some
function $f_0$. Then we have 
$$
\Big| \frac{f_n(z)}{P(z)} \e^{-\frac{1}{p} z^\alpha}\Big|
\le C_n \e^{a|z|^{\alpha/2}}, \qquad z\in \Omega,  
$$
for some constants $C_n$. We need to show that the constants $C_n$ can be
replaced by a  uniform bound. Once this is done, by taking the  pointwise
limit it follows that  $f_0$ satisfies the same estimate growth estimate
in the angle $\Omega$,  i.e. $f_0\in \mathcal{M}_a$.

Since $\sup_{n}\|f_n\|_{\fa^p} <+\infty$, it follows from \eqref{mir} that
there is a constant $A_0$ (independent of $n$) such that
$$
\Big| \frac{f_n(z)}{P(z)} \e^{-\frac{1}{p} z^\alpha}  \Big| \le A_0, \qquad z\in \Gamma.
$$
For $z\in \cp$, put 
$$
g_n(z) = \frac{\e^{z^2/p}f_n(z^{2/\alpha})}{P(z^{2/\alpha})}.
$$
Then $g_n$ is analytic in $\cp$ and continuous up to the boundary,
$|g_n(x)| \le A_0$, $x\in \rl$,
and
$$
|g_n(z)| \le C_n e^{a|z|}, \qquad z\in \cp.
$$
By the classical Phragm\'en--Lindel\"of principle, 
$|g_n(z)| \le A_0 e^{a \ima z}$, $z\in \cp$, whence for all $n$ we can replace
$C_n$ by $A_0$ and the result follows.
\end{proof}

One can also consider subspaces with $a<0$ if we replace the estimate
in the definition of $\mathcal{M}_a$ by two conditions:
$$
\Big| \frac{f(z)}{P(z)} \e^{-\frac{1}{p} z^\alpha} \Big| \le
C \e^{a|z|^{\alpha/2}}, \quad \arg z = \pi/\alpha,
\qquad 
\Big| \frac{f(z)}{P(z)} \e^{-\frac{1}{p} z^\alpha}\Big| \le
C \e^{ d |z|^{\alpha/2}}, \quad z\in \Omega,
$$
for some $d>0$. Also, in general, if $\alpha$ is sufficiently large, 
one can impose conditions in several angles of size $2\pi/\alpha$
in order to obtain further examples of closed  nontrivial nearly
invariant subspaces.

\begin{example}
{\rm In the case of the Fock space $\mathcal{F}_2$ one can impose conditions
in both half-planes:
$$
\mathcal{M}_{a,b} = \{f\in  \mathcal{F}_2:   |f(z)\e^{-z^2/2}| \le C\,\e^{b \ima z},
\ima z>0, 
\ \text{and}\ |f(z)\e^{-z^2/2}| \le C\e^{a \ima z}, \ima z<0\}.
$$
These subspaces actually coincide with those constructed in Subsection \ref{barg1}.}
\end{example}

\begin{example}
{\rm Consider the space $\mathcal{F}_1$ with the weight $W(z) = \exp(-|z|)$. 
Then $f\in \mathcal{M}_a$ if and only if $f\in \mathcal{F}_1$ and 
$$
|f(z) \e^{-z/2}| \le C \e^{a |z|^{1/2}}, \qquad z\in \co.
$$
Similar classes of functions appeared in the paper of Gurarii \cite{gur}
in connection to the study of primary ideals in weighted $L^1$ spaces
(Beurling algebras). }                         
\end{example}

Let us show that for any $\alpha>1$ the subspaces $\mathcal{M}_a$ in
$\fa^p$ are nontrivial. 

\begin{lemma} For any $\alpha\geq1$ and $a>0$ there exists entire function
$f\in\mathcal{F}^p_\alpha$ such that 
\begin{equation}
\Big| f(z)\e^{-\frac{z^{\alpha}}{p}}\Big| \leq  \e^{a |z|^{\alpha\slash2}}, \qquad
\arg z\in[0,2\pi\slash\alpha].
\label{efeq}
\end{equation}
\end{lemma}

\begin{proof}
If $\alpha = m \in \mathbb{N}$ and $a>0$, we can put
$$
f(z) = \e^{\frac{1}{p}z^m}\, \frac{\sin (a z^{m/2})}{z^{m/2}} (P(z))^{-1}
$$ 
for a polynomial $P$ whose zeros are a subset of the zero set of
$\sin (a z^{m/2})$. Then 
$f\in \mathcal{F}^p_m$ and
$$
\Big| f(z) \e^{-\frac{1}{2} z^m} \Big| 
\le C \e^{a|z|^{m/2}}, \qquad z\in\mathbb{C}.
$$

In the case when $\alpha$ is noninteger, the construction is not so explicit.
It is based on a subtle atomization theorem
of R. S. Yulmukhametov \cite{Yul}. Let $\Omega=\{z:\  0 <\arg z < 2\pi/\alpha\}$.
First we consider the case $a=0$ and construct
a nonzero function $f_1\in \mathcal{F}^p_\alpha$ such that 
\begin{equation}
|f_1(z) \e^{-\frac{z^{\alpha}}{p}}|\leq 1, \qquad z\in \Omega.
\label{eqef2}
\end{equation}
Fix a small number $\varepsilon \in \big(0, \pi\big(1-\frac{1}{\alpha}\big) \big)$
and put 
$$
h(\theta)=
\begin{cases} \cos \alpha\theta,
& \ \theta\in[0,2\pi\slash\alpha +\varepsilon]\cup[2\pi-\varepsilon,2\pi],
\\
\cos \alpha\varepsilon, & \ \theta\in [2\pi\slash\alpha +\varepsilon,
2\pi-\varepsilon ],
\end{cases}
$$
and 
$$
u(z)=\frac{1}{p}r^\alpha h(\theta).
$$
It is easy to see that $u$ is a subharmonic function in $\mathbb{C}$ (it is
harmonic in $\Omega$ and subharmonic in $\mathbb{C}\setminus \overline{\Omega}$). 
Next, we are going to use the following approximation result of \cite{Yul}.

\begin{theorem}
\label{yulmukh}
Let $u$ be a subharmonic function in the complex plane of finite order
$\rho$. Then there exists an entire function $F$
such that for every $\beta\ge \rho$  
$$
|\log|F(z)|-u(z)|\le  C_\beta\log|z|,\qquad z\in \mathbb{C}\setminus E_\beta,
$$
for some $C_\beta>0$, and the set $E_\beta$ can be covered by a family of
discs $D(z_j, r_j)$ such that
$$
\sum_{|z_j|>R} r_j=\ordo(R^{\rho-\beta}),\qquad R\rightarrow+\infty.
$$
\end{theorem}

Let $F$ be an entire function constructed from $u$ via this theorem
(here $\rho=\alpha$ and we can find a sufficiently large $\beta> \alpha$
to ensure that $\e^u$ is almost constant inside the exceptional discs
$D(z_j,r_j)$). Then it is easy to see that $f_1 = F\slash P$ is in
$\fa^p$ and satisfies inequality \eqref{eqef2} for some polynomial $P$.
The estimate remains valid inside the discs $D(z_j,r_j)$ by the maximum principle.

To construct a function in $\mathcal{M}_a$, $a>0$, but not in
$\mathcal{M}_{a'}$ for $a'<a$, we construct an entire function $f_2$
such that, for $z= r \e^{i\theta}$,
\begin{equation}
|f_2(z)| \leq 
\begin{cases} \e^{r^{\alpha\slash2}\sin(\alpha\theta\slash2)}, &
\theta\in[0,2\pi\slash\alpha], \\
1,  & \theta\notin[0,2\pi\slash\alpha],
\end{cases}
\label{eqef3}
\end{equation}
Then the function $f(z):=f_1(z)f_2(z)$ satisfies \eqref{efeq}, but not a
similar estimate with $a'<a$.

The function 
$$
g(z)=\frac{\sin (z^{\alpha\slash2})}{z^{\alpha\slash2}}
$$
is analytic in the sector $\Omega = \{z:\  0 <\arg z < 2\pi/\alpha\}$ and
bounded in $\partial\Omega$.
Put 
$$
g_1(z)=\frac{1}{2\pi i}\int_{\partial\Omega}\frac{g(w)}{z-w}dw,
\qquad z\in \mathbb{C} \setminus \overline{\Omega}.
$$
It is well known that $g_1$ extends to an entire function. Indeed, $g_1$ is
analytic in $\mathbb{C}\setminus\overline{\Omega}$. Let
$\Omega_R=\Omega \cap \{z:|z|>R\}$ and
$g_R(z)=\frac{1}{2\pi i}\int_{\partial\Omega_R}\frac{g(w)}{z-w}dw$,  $z\in 
\mathbb{C} \setminus \overline{\Omega}_R$.
Then the function $g_R$ is an analytic continuation of $g_1$ to $\mathbb{C}\setminus 
\overline{\Omega}_R$ and, thus, 
$g_1$ extends to an entire function. By the Sokhotski–Plemelj theorem we get
$$
|g_1(z)|\leq \e^{r^{\alpha\slash2}\sin(\alpha\theta\slash2)}+ \Ordo(r),
\quad \theta\in[0,2\pi].
$$
Using standard estimates of Cauchy transform we get $|g_1(z)|\leq 10 (1+|z|)$,
$z\not\in\Omega$.
Hence, function $f_2(z)=\frac{g_1(z)}{z-z_0}$ (where $z_0$ is some zero of $g_1$)
satisfies \eqref{eqef3}.
\end{proof}
\bigskip


\section{Counterexamples constructed via generating functions}
\label{count}

In this section we discuss an additional, quite  general method to construct
nontrivial nearly invariant subspaces. 

Let $\fw$ be a Fock-type space (not necessarily radial). In this section we
consider only Hilbertian case, though all arguments easily carry over to
$L^p$-setting. We denote by $k_z$ the reproducing kernel of $\fw$
at the point $z$.
In what follows we will assume that our weight has some mild regularity,  more
precisely,  there exists $N\ge 0$ 
such that
\begin{equation}
\label{bab1}
C_1 |z|^{-N} \le \frac{W(\zeta)}{W(z)} \le C_2 |z|^{N}, \qquad \zeta \in
D(z, |z|^{-N}), \ |z|>1.
\end{equation}

Now assume that there exists an entire function $G\in \fw$ such that
the zeros $\{t_n\}$ of $G$ are simple and power separated, i.e., 
\begin{equation}
\label{pow} 
{\rm dist}(t_n, \{t_m\}_{m\ne n}) \gtrsim (|t_n|+1)^{-N_1},
\end{equation}
and also its derivative at the zeros has almost maximal growth (again up to a
power factor):
\begin{equation}
\label{pow1} 
\|k_{t_n}\| \le (|t_n|+1)^{N_2}|G'(t_n)|.
\end{equation}
Here $N_1, N_2$ are some positive constants. 

All these assumptions are standard and natural. For many Fock-type spaces
(including standard spaces $\fa$, $\alpha>0$) one can choose the function
$G$ which approximates nicely 
the weight: there exists $N>0$ such that
\begin{equation}
\label{gen1} 
(|z|+1)^{-N} W^{-1/2} (z){\rm dist} (z, \zz_G) \lesssim |G(z)| \lesssim
(|z|+1)^{N} W^{-1/2}(z){\rm dist} (z, \zz_G) 
\end{equation}
for all $z\in\co$. 
Since, under condition \eqref{bab1} one has $\|k_z\| \lesssim (|z|+1)^N W(z)$
(for some, probably different, $N>0$), \eqref{pow1} follows from \eqref{gen1}.

In the classical Fock space (with $\varphi(z) = \pi |z|^2$) the corresponding
function $G$ is the 
the Weierstrass function 
$$
\sigma(z) = z\prod_{\lambda \in (\zl+i\zl)\setminus\{0\}} \bigg(1-\frac{z}
{\lambda}\bigg)e^{\frac{z}{\lambda} +\frac{z^2}{2\lambda^2}}.
$$
In general, for a subharmonic weight $W$ the function $G$ can be obtained
by atomization of $\Delta W$
(see, e.g., \cite{Yul} or \cite[Proposition 8.1]{bdk}).

\begin{theorem}
\label{const}
Let $\fw$ and $G$ satisfy conditions \eqref{bab1}, \eqref{pow} and \eqref{pow1}.
Assume that there exists an entire function $T$ such that $\zz_{T} \subset \zz_G$, 
both $\zz_{T}$ and $\zz_G\setminus \zz_{T}$ are infinite, and the following holds:
\begin{equation}
\label{t1} 
|T(z)| \gtrsim (|z|+1)^{-N}, \qquad z \notin \cup_n D_n, 
\end{equation}
where $D_n = D(t_n, (|t_n|+1)^{-N}/10)$, and
\begin{equation}
\label{t2} 
|T(t_n)| \lesssim 1, \qquad t_n\in \zz_G\setminus\zz_{T}.
\end{equation}
Then $\fw$ contains a nontrivial infinite-dimensional subspace of
the form $\mathcal{M}_g$ for some $g\in \fw$.
\end{theorem}

\begin{proof}
We may assume that the constants $N$, $N_1$, $N_2$ in 
\eqref{bab1}--\eqref{t2} are the same. 
Dividing $T$ by a polynomial of sufficiently large degree with zeros in
$\zz_{T}$ we may assume, 
without loss of generality, that
\begin{equation}
\label{sho} 
\sum_{n}\frac{|T(t_n)|}{|G'(t_n)|}\|k_{t_n}\| <+\infty.
\end{equation}
Next we construct an entire function $G_0$ whose zeros $\lambda_n = t_n+\delta_n$, 
$t_n\in \zz_G\setminus \zz_{T}$, are very small but {\it nonzero} perturbations
of zeros of $G/T$ so that
\begin{equation}
\label{as} 
|G_0(z)T(z)| \asymp |G(z)|,  \qquad z \notin \cup_n D_n.
\end{equation}
and 
\begin{equation}
\label{as1} 
\sum_n \Big|\frac{G_0(t_n) T(t_n)}{G'(t_n)}\Big|<+\infty
\end{equation}

Now let $P$ be a polynomial with $\zz_P\subset \zz_{G_0}$. Put
$g = \frac{G_0}{P}$. We claim that if the degree of $P$
is sufficiently large, then $g\in \fw$ and
$\mathcal{M}_g = \overline{{\rm Span}} \big\{ \frac{g(z)}{z-\lambda}: \
\lambda \in \zz_g \big\}$
is a nontrivial nearly invariant subspace.
We have
$$
|g(z)| \asymp \frac{|G(z)|}{|P(z)T(z)|}\lesssim
\frac{|z|^{N}G(z)}{|P(z)|}, \qquad z\notin \cup_n D_n.
$$
Thus, if ${\rm deg} P \ge N$, then $W^{-1/2}g \in L^2(\co \setminus \cup_n D_n)$.
Since the weight $W$ satisfies 
\eqref{bab1} and the discs $D_n$ are disjoint, we conclude that $g\in \fw$ if
${\rm deg} P$ is sufficiently large.

Now we construct a function which is orthogonal to $\MM_g$. Let
$b_n = T(t_n) \|k_{t_n}\|/G'(t_n)$ and let
$$
f = \sum_n \overline{b_n} k_{t_n}.
$$ 
In view of \eqref{sho} $f\in\fw$. Clearly, for $\lambda\in \zz_g$,
$$
\bigg( \frac{g(z)}{z-\lambda}, f \bigg)_{\fw} = \sum_n b_n
\frac{g(t_n)}{t_n-\lambda} =\sum_n \frac{g(t_n)T(t_n)}{G'(t_n)(t_n-\lambda)}
$$
(note that $\lambda\ne t_n$ by the construction of $G_0$). We claim that
$$
\sum_n \frac{g(t_n)T(t_n)}{G'(t_n)(z- t_n)} = \frac{g(z)T(z)}{G(z)}.
$$
Then, putting $z= \lambda$ yields that $\big(\frac{g(z)}{z-\lambda},f\big)_{\fw} =0$.

Let 
$$
H(z) = \sum_n \frac{g(t_n)T(t_n)}{G'(t_n)(z- t_n)} - \frac{g(z)T(z)}{G(z)}.
$$
Since the residues coincide, $H$ is an entire function. Note that 
$$
\bigg|\frac{g(z)T(z)}{G(z)}\bigg| = \bigg|\frac{G_0(z)T(z)}{G(z)}
\bigg|\cdot \frac{1}{|P(z)|}\lesssim \frac{1}{|P(z)|}, \qquad z\notin \cup_n D_n.
$$ 
Combining this with \eqref{as1} we conclude that $|H(z)|\lesssim 1+|z|^N$ and
also $\liminf_{|z|\to+\infty}|H(z)| =0$. Thus, $H\equiv 0$.
\end{proof}

Here are some concrete examples where  Theorem \ref{const} applies.

\begin{example}
{\rm All examples of Subsection \ref{barg1} can be obtained by the construction of Theorem \ref{const}.
E.g., let $W(z) = \exp(-\pi |z|^2)$ and $\MM = \e^{\pi z^2/2}PW_{[-\pi, \pi]}$. 
Let $G=\sigma$ be the Weierstrass function and
$T(z) = \frac{\sigma(z) e^{-\pi z^2/2}}{\sin \pi z}$.
Then the subspace constructed in Theorem \ref{const} starting from these $G$ and $T$ coincides with $\MM$.
In the case of $\MM =\e^{\pi z^2/2}PW_{[-a, a]}$ one should replace
$\sin \pi z$ by an appropriate product 
of sine-type functions with zeros in the lattice $\mathbb{Z} + i\mathbb{Z}$ intersected with a strip around $\rl$.}
\end{example}

\begin{example}
{\rm A different example in the classical Fock space $\mathcal{F}_2$ with
$W(z) = \exp(-\pi |z|^2)$  
can be obtained if we consider 
$$
T(z) = \prod_{k=1}^{+\infty} \bigg(1-\frac{\e^{2\pi iz}}{\e^{2\pi k}}\bigg).
$$
Then $T$ is an entire function with zeros $z_{m,k} = m-i k$, $m\in\mathbb{Z}$,
$k\in \mathbb{N}$, and it is easy to see that 
$T$ satisfies \eqref{t1}--\eqref{t2} with $G=\sigma$. This function was used in 
\cite{abb} to provide counterexamples to the  Ordering Theorem
for nearly invariant subspaces in Cauchy--de Branges spaces. 
The corresponding nearly invariant subspace coincides with the
subspace $\mathcal{M}_0$ from Proposition \ref{phl}. }
\end{example}

\begin{example}
{\rm Consider $\mathcal{F}_1$, i.e., let $W(z) = \exp(-|z|)$. Then, by an
atomization procedure one can find an
entire function $G$ satisfying \eqref{gen1} (see, e.g., \cite{Yul}
or \cite[Proposition 8.1]{bdk}). 
Moreover, it is easy to choose its zeros $t_n$ so that $\{t_n\} \supset 
\{ \beta^{-1} n^2 \}_{n\ge 1}$ for some $\beta>0$. Then we set
$$
T(z) = G(z)\e^{-z} \bigg( \frac{\sin \pi\beta \sqrt{z}}{\sqrt{z}}\bigg)^{-1}.
$$
Conditions \eqref{t1}--\eqref{t2} are clearly satisfied.}
\end{example}
\bigskip


\section{Rotation-invariant nearly invariant subspaces}
\label{rot}

Recall that these are nearly invariant subspaces of $\fw^p$ ($W$ radial),
which are also invariant under the isometry $R_\beta$,  defined by
$R_\beta f(z)=f(\e^{\imag\beta}z),~f\in \fw^p$. 
As we mentioned in the introduction, in the case when $\beta/\pi \notin \mathbb{Q}$,
such a subspace must have the form $\mathcal{P}_k$, for some integer $k\ge 0$.

To verify this, note that the symmetrizations defined for integers $k\ge 0$,
$N\ge 1$, by
\begin{equation}
\label{ser} 
R_{k,N}=\frac{1}{N} \sum_{j=0}^{N-1}\e^{-\imag jk\beta}R_\beta^j,
\end{equation}
are uniformly bounded on $\fw^p$, and if $f$ is a polynomial, then $R_{k,N}f$
converges in the norm of $\fw^p$ to the monomial 
$\frac{f^{(k)}(0)}{k!} z^k$ as $N\to+\infty$.
By Proposition \ref{den}, the polynomials are dense
in $\fw^p$, hence the above assertion holds true for every $f\in \fw^p$. Thus, if
$\MM\subset\fw^p$ is a closed nearly invariant subspace without common zeros,
then it contains the monomial $z^k$ whenever there exists $f\in\MM,~f^{(k)}(0)\ne 0$.
In particular, it contains the constants which makes $\MM$ backward shift
invariant. But then,  if $\MM$  contains the monomial 
$z^k$, it will also contain $z^l$, $0\le l\le k$. If $\MM$ contains all the
monomials, then $\MM=\fw^p$, given that the polynomials are dense. We conclude that
if $\MM$ is a nontrivial nealy invariant subspace, then there must exist $k_0$
such  that $f^{(k)}(0)=0$ whenever $f\in \MM$, $k\ge k_0$. From this, the
claim is immediate.

\begin{proof}[Proof of Theorem \ref{simetr}]
Throughout the proof we assume that $\beta = 2\pi /n$. Given an $R_\beta$-invariant 
subspace $\MM$ of $\fw^p$, we denote by $\MM^s$
the set of all functions $F$ in $\MM$ such that $R_\beta F = F$.  In other words,
with the notation in  \eqref{ser}, $\MM^s=R_{0,n}\MM$. 

Now let $\MM_1$ and $\MM_2$ be two nearly invariant subspaces invariant also
with respect to $R_\beta$. We will show that either $\MM_1^s \subset\MM_2^s$
or $\MM_2^s \subset\MM_1^s$. As in the proof of Theorem \ref{main3}, we fix
nonzero functions
$F_1, F_2 \in L^q(W)$ such that $\Psi_{F_1}$ annihilates $\MM^s_2$ but
not $\MM^s_1$, while $\Psi_{F_2} $ annihilates $\MM_1^s$ but
not $\MM_2^s$.
Note that, if $R_\beta F = F$, then  
$$
\begin{aligned}
\Psi_{F_1}(F) & = \frac{1}{n} \sum_{j=0}^{n-1} \int_{\co}
F(\e^{\imag j \beta} z) F_1(z) W(z) dm_2(z) \\
& =
\frac{1}{N} \int_{\co}  F(z)\sum_{j=0}^{N-1} F_1(\e^{-\imag j \beta} z)
W(z) dm_2(z) =\Psi_{F_1^s}(F).  
\end{aligned}
$$
Thus, we may assume without loss of generality that $R_\beta F_1 = F_1$,
$R_\beta F_2 = F_2$.

Define entire functions $f$ and $g$, by the formulas \eqref{ff} and \eqref{fg},
respectively. Then, for $\mu = W m_2$, 
$$
\begin{aligned}
f(\e^{\imag\beta} z) &= \int\frac{F(\zeta) -
\frac{F(\e^{\imag\beta}z)}{G(\e^{\imag\beta}z)}G(\zeta)}{\zeta-\e^{\imag\beta}z} 
F_1(\zeta) d\mu(\zeta) \\
& = \int\frac{F(\e^{\imag\beta}u) - \frac{F(\e^{\imag\beta}z)}{G(\e^{\imag\beta}z)}
G(\e^{\imag\beta}u)}{\e^{\imag\beta}u-\e^{\imag\beta}z} 
F_1(\e^{\imag\beta}u)  d\mu(u) =  \e^{-\imag\beta}f(z).  
\end{aligned}
$$
Thus, $\tilde f(z) = z f(z)$ satisfies $\tilde f(\e^{\imag\beta} z) =\tilde f(z)$,
whence $\tilde f (z) = f_1(z^n)$. 
Since all functions in $\fw^p$ are of zero type 
with respect to the order $n$, the same is true for the function $f$
(see \cite[Chapter 1, \S 9]{lev}). We 
conclude that $f_1$ is a function of zero exponential type. Analogously, 
$zg(z) = g_1(z^n)$, where $g_1$ is of zero exponential type. 

By \eqref{ct}, there is a set $\Omega$ of finite measure such that 
$$         
|f_1(z^n)| \lesssim |z|^3\bigg(1+ \bigg|\frac{F(z)}{G(z)}\bigg|\bigg), \qquad
|g_1(z^n)| \lesssim |z|^3 \bigg(1+ \bigg|\frac{G(z)}{F(z)}\bigg|\bigg),
\qquad |z|>1, z\notin \Omega.
$$
Hence,  $\min (|f_1(z)|, |g_1(z)|) \lesssim |z|^3$ outside $\Omega$ 
and, by the Heins Theorem \ref{hei}, 
either $f_1$ or $g_1$ is a polynomial. Repeating the arguments from the proof
of Theorem \ref{main3} we arrive at a contradiction. 

We have shown that the symmetrized parts of nearly invariant rotation
invariant subspaces are ordered by inclusion. In particular, for any
infinite-dimensional rotation invariant subspace $\MM$, its symmetrized part
$\MM^s$ must contain the collection of polynomials 
of the form $\sum_{k=0}^m p_k z^{kn}$, 
the symmetrized part of $\pp_{mn}$. Since $\MM$ is also nearly invariant, 
we conclude that $\MM$ contains all the polynomials.
If $\MM$ is instead finite-dimensional, there exists 
$m$ such that $\pp_{m}^s=\MM^s$ and so $\MM = \pp_k$ for some $k$.
\end{proof}


\begin{thebibliography}{BRSHZE}

\bibitem{abb}  E.~Abakumov, A.~Baranov, Yu.~Belov, 
{\it Krein-type theorems and ordered structure for Cauchy--de Branges spaces}, 
J. Funct. Anal. {\bf 277} (2019), 1, 200--226.

\bibitem{ar}  A.~Aleman, S.~Richter,
{\it Simply invariant subspaces of $H^2$ of some multiply 
connected regions}, Integr Equat. Oper. Theory {\bf 24} (1996), 2, 127--155.

\bibitem{ar1} A. Aleman, W. T. Ross, {\it The backward shift on weighted Bergman
spaces}, Michigan Math. J. {\bf 43} (1996), 2, 291--319.

\bibitem{arr} A. Aleman, S. Richter, W.T. Ross, {\it Pseudocontinuations and
the backward shift},  Indiana Univ. Math. J. {\bf 47} (1998), 1, 223--276.

\bibitem{ars}  A.~Aleman, S.~Richter, C.~Sundberg, 
{\it Invariant subspaces for the backward shift on Hilbert spaces of analytic
functions with regular norm}, Contemporary Mathematics, vol. 404, 2006, 1--25.

\bibitem{by1} A.~Baranov,
{\it Spectral theory of rank one perturbations of normal compact operators},
Algebra i Analiz {\bf 30} (2018), 5, 1--56; English transl. in St.
Petersburg Math. J. {\bf 30} (2019), 5, 761--802.  
  
\bibitem{bbb-fock} A.~Baranov, Yu.~Belov, A.~Borichev, 
{\it Summability properties of Gabor expansions}, 
J. Funct. Anal. {\bf 274} (2018), 9, 2532--2552.

\bibitem{bbb-arxiv}
A. Baranov, Y. Belov, A. Borichev, {\it Summability properties of Gabor expansions}, 
Version 2. Dec. 5, 2018. arXiv:1706.05685v2.

\bibitem{bom} A. Baranov, H. Bommier-Hato, 
{\it De Branges spaces and Fock spaces}, 
Complex Var. Elliptic Equ. {\bf 63} (2018), 7--8, 907--930.

\bibitem{by} A.~Baranov, D.~Yakubovich,
{\it Completeness and spectral synthesis 
of nonselfadjoint one-dimensional perturbations of selfadjoint operators}, 
Adv. Math. {\bf 302} (2016), 740--798.

\bibitem{bdk} 
A. Borichev, R. Dhuez, K. Kellay,
{\it Sampling and interpolation in large Bergman and Fock spaces}, J. Funct. Anal.
{\bf 242} (2007), 2, 563--606.

\bibitem{br} L.~de Branges, \emph{Hilbert Spaces of Entire
Functions}, Prentice--Hall, Englewood Cliffs, 1968.

\bibitem{dom}
Y. Domar, {\it On the unicellularity of $\ell^l(w)$}, Monatsh. Math. {\bf 103} (1987), 2, 103--113.

\bibitem{dur}
P. Duren, {\it Theory of $H^p$ Spaces},  Academic Press, New York, 1970.

\bibitem{gur}
V. P. Gurarii, {\it Harmonic analysis in spaces with a weight}, Trudy Mosk. Mat.
Obshch. {\bf 35} (1976), 
21--76; English transl. in Trans. Moscow Math. Soc. {\bf 35} (1979), 21--75.

\bibitem{hay} E. Hayashi, {\it The kernel of a Toeplitz operator}, Integr. Equat.
Oper. Theory {\bf 9} (1986), 4, 588--591.
                                
\bibitem{heins}  
M. Heins, {\it On a notion of convexity connected with a method of Carleman},
J. Anal. Math. {\bf 7} (1959), 53--77.

\bibitem{hitt} D. Hitt, {\it Invariant subspaces of $H^2$ of an annulus},
Pacific J. Math. {\bf 134} (1988), 101--120.

\bibitem {lev} B.~Ya.~Levin, {\it Distribution of Zeros
of Entire Functions}, GITTL, Moscow, 1956;
English translation: Amer. Math. Soc., Providence, 1964;
revised edition: Amer. Math. Soc., 1980. 

\bibitem{mm} P. Mattila, M. S. Melnikov,
{\it Existence and weak-type inequalities for Cauchy integrals 
of general measures on rectifiable curves and sets}, 
Proc. Amer. Math. Soc. {\bf 120} (1994), 143--149. 

\bibitem{nik} N. K. Nikol'skii, {\it Basicity and unicellularity of
weighted shift operators}, Izv. Akad. Nauk SSSR Ser. Mat. {\bf 32} (1968),
5, 1123--1137; English transl. in Math. USSR-Izv. {\bf 2} (1968), 5, 1077--1089.

\bibitem{ri} S. Richter, {\it A representation theorem for cyclic analytic
two-isometries}, Trans. Amer. Math. Soc. {\bf 328} (1991), 325--349.

\bibitem{risu} S. Richter, C, Sundberg, {\it Invariant subspaces of the
Dirichlet shift and pseudocontinuations}, Trans. Amer. Math. Soc. {\bf 341}
(1994), 863--879.

\bibitem{rom} R.~Romanov, {\it Canonical systems and de Branges spaces}, 
arXiv:1408.6022.

\bibitem{sar}
D. Sarason, {\it Nearly invariant subspaces of the  backward 
shift}, Oper. Theory: Adv. Appl. {\bf 35} (1988), 481--493.

\bibitem{shi}
A.~L.~Shields, {\it Weighted shift operators and analytic function theory}, Topics in Operator Theory, Math. Surveys 13, Amer. Math. Soc., 
Providence, R.I., 1974, 49–128.
 
\bibitem{verd} J. Verdera, {\it A weak type inequality for Cauchy transforms 
of finite measures}, Publ. Math. {\bf 36} (1992), 1029--1034.

\bibitem{ya1}
D.~V.~Yakubovich, {\it Conditions for unicellularity of weighted shift operators}, 
Dokl. Akad. Nauk SSSR {\bf 278} (1984), 4, 821--824.

\bibitem{ya2}
D.~V.~Yakubovich, {\it Invariant subspaces of weighted shift operators}, 
Zap. Nauchn. Sem. Leningrad. Otdel. Mat. Inst. Steklov. (LOMI) {\bf 141} (1985), 100--143;
English transl. in:  J. Soviet Math. {\bf 37} (1987), 5, 1323--1346.

\bibitem{Yul} R.~S.~Yulmukhametov, {\it Approximation of subharmonic functions}, 
Anal. Math. {\bf 11} (1985) 257--282.
\end{thebibliography}
\end{document}